\documentclass{amsart}
\usepackage{amssymb,amsmath,amsthm,latexsym}
\usepackage[all]{xy}
\usepackage{amssymb,graphicx}
\newtheorem{theorem}{Theorem}

\newtheorem{corollary}[theorem]{Corollary}

\newtheorem{definition}[theorem]{Definition}

\newtheorem{lemma}[theorem]{Lemma}

\newtheorem*{corollary2}{Corollary \ref{firstcorollary}}
\newtheorem*{GeomSuperRigidity}{The Geometric Superrigidity Theorem}
\newtheorem*{TopSuperRigidity}{The topological superrigidity theorem}
\newtheorem*{PDtheorem}{Theorem \ref{PD}}
\renewcommand{\:}{\mathord{:}\hspace*{1ex plus .25ex minus .3ex}}

\newcommand{\field}{\mathbb{F}_2}
\renewcommand{\hom}{\operatorname{Hom}}
\DeclareMathOperator{\CAT}{CAT}
\newcommand{\Z}{\mathbb{Z}}
\newcommand{\cohom}[3]{\operatorname{H}^{#1}({#2}, {#3})}
\newcommand{\ext}[4]{\operatorname{Ext}^{#1}_{#2}({#3}, {#4})}

\renewcommand{\:}{\mathord{:}\hspace*{1ex plus .25ex minus .3ex}}
\def\coind{\textrm{Coind}}
\def\ind{\textrm{Ind}}
\def\hom{\textrm{Hom}}
\begin{document}
\title{A topological splitting theorem for Poincar\'e duality groups and high dimensional manifolds}
\author{Aditi Kar \and Graham A. Niblo}
\address{Mathematical Institute, University of Oxford, 24-29 St Giles', Oxford OX1 3BL, England.}
\address{School of Mathematics, University of Southampton, Highfield, Southampton, SO17 1BJ, England.}
\email{Aditi.Kar@maths.ox.ac.uk}
\email{G.A.Niblo@soton.ac.uk}
\thanks{The first author is by EPSRC grant EP/I020276/1 and the second author is partially funded by a senior Leverhulme Fellowship.}

\begin{abstract}
We show that for a wide class of manifold pairs $N,M$ satisfying $\dim(M)=\dim(N)+1$, every $\pi_1$-injective map $f:N\rightarrow M$ factorises up to homotopy as a finite cover of an embedding. This result, in the spirit of Waldhausen's torus theorem, is derived using Cappell's surgery methods from a new algebraic splitting theorem for Poincar\'e duality groups. As an application we derive a new obstruction to the existence of $\pi_1$-injective maps.
\end{abstract}
\maketitle

\section{Introduction}
\label{intro}
\textit{\textbf{Conventions}: We use superscripts to denote real dimension e.g. a manifold denoted $N^{k}$ has dimension $k$. Once the dimension is established we omit the superscript so that the manifold $N^{k}$ is also denoted $N$.  A group $G$ is said to \emph{split over a subgroup $H$} if $G$ has one of the following descriptions:\\
\noindent Case I (Free product with amalgamation) $G=G_1 *_H G_2$ with $G_1\not = H\not = G_2$, \\
\noindent Case II (HNN extension) $G=J*_H$, $J \neq H$.}
\medskip

The presence of a group action on a space often allows one to promote an existing structure to another of a more strictly controlled kind. Examples of this phenomenon in topology include Papakyriokopoulos's sphere theorem \cite{Papakyriokopoulos}, Waldhausen's torus theorem \cite{Waldhausen}, and the geometric superrigidity theorem of Mok {\it et al}, \cite{Mok}. In this paper we propose the following topological result which has features in common with them all. 

\begin{theorem} \label{main} Let $N^{n}$ be a closed, orientable, aspherical topological manifold with $n$ even and $n\geq 6$, satisfying the following properties:
\begin{enumerate}
\item\label{Fbox} every cellular action of $\pi_1(N)$ on a CAT(0) cubical complex has a global fixed point,
\item\label{K_0} the projective class group $\widetilde{K_0}(C)$ vanishes for any torsion free finite extension $C$  of $\pi_1(N)$.
\end{enumerate}

\noindent Given any closed, orientable, aspherical topological manifold $M^{n+1}$ and  any $\pi_1$-injective, continuous function $j: N \rightarrow  M$,
there is a diagram as follows which commutes up to homotopy

\begin{figure}[h]
\label{diagram1}
\begin{center}
$\xymatrix{N_0\ar@{-->}_p[d] &N \ar@{->}_j[dd] \ar@{<-->}[l]_{{h'}} \\ N'  \ar@{-->}_i[d] &   \\M' \ar@{<-->}[r]^{h} & M }$
\end{center}
\end{figure}

\noindent where:\begin{enumerate}
\item  $M', N_0$ are closed, orientable, aspherical, topological  manifolds and $h, h'$ are  homotopy equivalences,
\item The map $p:N_0\rightarrow N'$ is a finite degree cover, and
\item $i:N'\rightarrow M'$ is a two sided topological embedding. 
\end{enumerate}
\end{theorem}

The reduced projective class group is conjectured to vanish for all torsion free groups (see \cite[Conjecture H.~1.~8] {Davis} for a discussion) so Condition (2) of  Theorem is conjecturally unnecessary. Condition (1) on the other hand is at the heart of the result. By Sageev's duality theorem \cite{Sageev} together with Scott's ends theorem \cite{Scott} it is equivalent to the assertion that no covering space of $N$ has more than 1 topological end. This is central to the proof which proceeds in 2 stages, first to derive a splitting of the fundamental group of $M$ as an HNN extension or as a non-trivial amalgamated free product and then to use the splitting to obtain the required decomposition of $M$. The first stage is essentially group theoretic and requires condition 1, while the second uses surgery theoretic techniques which necessitates the restriction to dimensions $\geq 5$. The restriction to even dimensions is unnecessary; however, the surgery methods required for the odd dimensional case are more 
complicated and we address them in a forthcoming paper. 

Subject to the Borel conjecture, there is a more elegant formulation of Theorem \ref{main}. The Borel conjecture, formulated by Armand Borel in 1953 asserts that any homotopy equivalence between  aspherical manifolds should be homotopic to a homeomorphism. Since the homotopy type and the dimension of  an aspherical manifold is determined by its fundamental group the Borel conjecture asserts that the fundamental group is sufficient to classify aspherical manifolds of any given dimension.

If in Theorem \ref{main}, both $\pi_1(M)$ and $\pi_1(N)$ are known to satisfy the Borel conjecture, or, alternatively,  if $\pi_1(N)$ satisfies the Borel conjecture and is square root closed in $\pi_1(M)$, then the homotopy equivalences $h,h'$  are homotopic to homeomorphisms and  the commutative diagram simplifies as below. In this case the theorem asserts that every $\pi_1$-injective codimension-$1$ map $j:N\rightarrow M$ factors, up to homotopy, as a finite cover of a $2$-sided embedding.
 
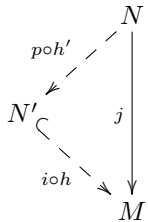
\begin{figure}[h]
\label{diagram2}
\begin{center}
$\xymatrix{&N \ar@{->}_j[dd] \ar@{-->}[ld]_{p\circ h'} \\ N' \ar@{^{(}-->}[rd]_{i\circ h} & \\ &M }$
\end{center}
\caption{When $\pi_1(M)$ and $\pi_1(N)$ satisfy the Borel conjecture the map $j$ factors up to homotopy as a finite cover, $p\circ h'$,  of an embedding $i\circ h$.}
\end{figure}

\label{superrigid} The hypotheses on $N$ are rigidity constraints which, for example, are satisfied when $N$ is a closed, orientable, Riemannian manifold whose universal cover $\tilde N$ is quaternionic hyperbolic or the Cayley hyperbolic plane. More generally, by \cite{NibloReeves}, condition \ref{Fbox} is satisfied whenever $\pi_1(N)$ satisfies Kazhdan's property (T). By results of Bartels and L\"uck \cite{BL}, condition \ref{K_0} is satisfied whenever $\pi_1(N)$ is a word hyperbolic group. Indeed, any finite extension of a hyperbolic group is itself, hyperbolic and so \cite{BL} applies. For the following reasons condition \ref{K_0} is also satisfied when $\pi_1(N)$ is CAT(0), i.e. \emph{it admits a co-compact isometric proper action on a finite dimensional CAT(0) space}.  

Assume that $H=\pi_1(N)$ admits a co-compact, proper isometric action on a CAT(0) space $X$ and that  $C$ is a finite extension of $H$. Letting $K$ denote the intersection of the conjugates of $H$ in $C$ we exhibit $C$ as a finite extension 

$$1 \rightarrow K \rightarrow C \rightarrow Q \rightarrow 1.$$

\noindent By the universal embedding theorem \cite[Theorem 2.6A]{dixonmortimer}, $C$ embeds in the standard wreath product $K \wr Q$ which acts cocompactly and properly discontinuously by isometries on the  the $Q$-fold direct product $X^Q$ of $X$ equipped with the induced CAT(0) metric. Hence, by \cite{BL} the Farrell-Jones Conjecture holds for the wreath product $K\wr Q$. However, as the property of a group satisfying the Farrell-Jones conjecture is preserved under taking subgroups, we conclude that $C$ satisfies the Farrell-Jones conjecture. Consequently the projective class group vanishes for every torsion-free finite extension of $\pi_1(N)$. 	

Hence we obtain as a corollary:
 
\begin{TopSuperRigidity}\label{superigid} Let $M^{n+1}$ and $N^n$ be closed, aspherical, orientable topological manifolds with $n$ even and $n\geq 6$ such that  $\pi_1(N)$ is either a word-hyperbolic or a CAT(0) group that satisfies Kazhdan's property (T).  If $M$ satisfies the Borel conjecture then every $\pi_1$-injective map $j:N\rightarrow M$ is, a finite cover of an embedding (up to homotopy). 
\end{TopSuperRigidity}

We regard this as a topological counterpart to the celebrated Geometric Superrigidity theorem:

\begin{GeomSuperRigidity}[Ngaiming Mok, Yum-Tong Siu, Sai-Kee Yeung, \cite{Mok}]\label{superrigidity}
Let $\tilde N$ be a globally symmetric irreducible Riemann manifold of non-compact type. Assume that either $\tilde N$ is of rank at least $2$, or  is the quaternionic hyperbolic space of dimension at least $8$ or the hyperbolic Cayley plane. Let $H$ be a cocompact discrete subgroup of the group of isometries of $\tilde N$ acting freely. Let $\tilde M$ be a Riemann manifold. Let $f$ be a non-constant $H$-equivariant harmonic map from $\tilde N$ to $\tilde M$. When the rank of $\tilde N$ is at least $2$, the Riemann sectional curvature is assumed to be non-positive. When the rank of $\tilde N$ is one, the complexified sectional curvature is assumed to be nonpositive. Then the covariant derivative of the differential of $f$ is identically zero. As a consequence, $f$ is a totally geodesic isometric embedding (up to a renormalization constant).
\end{GeomSuperRigidity}

To compare the two results note that the map $j:N\rightarrow M$ induces a $\pi_1(N)$ -equivariant map between the universal covers. The constraint on the curvature of the target $\tilde{M}$ in the geometric super rigidity theorem is dropped (together with the requirement that the target is smooth) in favour of a statement that $\tilde{M}$ is contractible and $\pi_1(M)$ satisfies the Borel conjecture. Instead of a harmonic map, we start with a continuous function which is  codimension-1. The conclusion that the map is  totally geodesic up to renormalisation is replaced by the conclusion that it is a finite cover of an embedding up to homotopy. Deforming the map $j$ to a harmonic map in the same homotopy class allows us to combine the conclusions of the two  results to see that the embedding provided by Theorem \ref{main} is homotopic to a totally geodesic surface in $M$.

The reader may find it helpful in visualising the results in this paper to consider the analogous statements in lower dimensions. First we consider the case of $\pi_1$-injective loops on a $2$-torus $T$. The fundamental group $\pi_1(T)$ is free abelian of rank $2$ and so any homotopy class of curves is represented by a pair of integers $(m,n)$. It is an elementary fact that a non-trivial curve is homotopic to a simple closed curve if and only if the pair $(m,n)$ is coprime, and it follows that in general a curve representing a pair $(a,b)$ is homotopic to a degree $d$ cover of the embedded curve representing the pair $(a/d, b/d)$ where $d=\gcd(a,b)$. It follows that every closed $\pi_1$-injective curve is homotopic to a finite cover of an embedded loop.  This is a direct analogue of the codimension-$1$ topological rigidity theorem.

Now, in contrast, consider the case of loops on a \emph{hyperbolic} surface. Every orientable hyperbolic surface $\Sigma$ admits $\pi_1$-injective maps $\gamma:S^1 \rightarrow \Sigma$ which do not factorise up to homotopy as a finite cover of an embedding. Recall that the free homotopy class of a closed loop contains a unique geodesic, and that this minimises the self intersection number for curves in that class. On the other hand, intersection numbers are multiplicative on powers, and it follows that any self intersecting closed geodesic which has intersection number $1$ with some simple closed curve provides a loop which does not finitely cover an embedded loop so the analogous statement fails in this case. Such surfaces do carry many splitting curves, and these can often be obtained from immersed totally geodesic curves by the somewhat different methods of cut and paste.

In dimension $3$ the situation worsens: the Kahn and Markovic theorem, \cite{KahnMarkovic}, shows that every hyperbolic $3$-manifold contains an immersed $\pi_1$-injective surface, however, there are examples which do not contain any embedded $\pi_1$-injective surfaces which they could cover. For example, in \cite{Reid}, Reid constructs a non-Haken manifold $M^3$ which admits a finite cover homeomorphic to a hyperbolic surface bundle over $S^{1}$. The fibre in the finite cover yields an immersed surface in $M$ but there are no embedded surfaces in $M$ which it could cover so there is in general no  analogue of the codimension-$1$ topological rigidity theorem in low dimensions. This leaves open the question of what happens in dimension $n=4$.

\bigskip

The strategy in the proof of Theorem \ref{main} is to first establish the existence of a splitting of the fundamental group $\pi_1(M)$ over a suitable subgroup using methods from geometric group theory. This result applies in the context of Poincar\'e duality groups and is of independent interest.

\begin{theorem}\label{PD}
Let $G$ be an orientable $PD^{n}$ group and $H$ be an orientable $PD^{n-1}$ subgroup of $G$. If every action of $H$ on a CAT(0) cube complex has a global fixed point (in particular, by \cite{NibloReeves},  if $H$ satisfies Kazhdan's property (T)), then $G$ splits over a subgroup $C$ containing $H$ as a finite index subgroup. 
\end{theorem}

Theorem \ref{PD} may be viewed as an analogue of the algebraic annulus and torus theorems of Kropholler and Roller \cite{KrophollerRoller}, Dunwoody and Swenson \cite{Dunwoody-Swenson}. In a companion paper we give a group theoretic application of Theorem \ref{PD}, obtaining a canonical decomposition of Poincar\'e duality groups over codimension-1 property (T) Poincar\'e duality subgroups. This may be viewed as an analogue of the algebraic JSJ decomposition studied by Kropholler \cite{Kropholler}, Dunwoody and Sageev \cite{Dunwoody-Sageev}, Scott and Swarup \cite{Scott-Swarup} and others.  

While Theorem \ref{main} is stated in the topological category the surgery technology applied in this paper also works in the smooth category and if the map $j$ is smooth then  the map $i$ it constructs is smooth also. We will use this fact to deduce the following obstruction result from Theorem \ref{main}.

\begin{corollary}\label{firstcorollary}
Let $M^{4d+1}$ be a closed, orientable, aspherical, smooth manifold  such that $d\geq2$ and the first Betti number $b_1(M)$ is zero. Let $N^{4d}$ be a closed, orientable, aspherical, smooth manifold with at least one non-zero Pontryagin number such that $\pi_1(N)$ is either word hyperbolic or a CAT(0) group and satisfies Kazhdan's property (T). Then there are no $\pi_1$-injective continuous maps $f\: N \rightarrow M$.
\end{corollary}

The proof of Theorem \ref {main} relies on ideas from geometric group theory, surgery theory, homological algebra and rigidity theory and the paper is organised as follows. In Section \ref{prelim1} we set up and prove the algebraic splitting theorem, Theorem \ref{PD}, for Poincar\'e duality groups. In section \ref{prelim2} we outline  the topological ingredients required for the proof of the topological splitting theorem, Theorem \ref{main}. In Section \ref{proof} we present the proof of Theorem \ref{main}, and in section \ref{application} we give the proof of the obstruction result, Corollary \ref{firstcorollary}.

\section{Splitting Poincar\'e duality groups} \label{prelim1}
In order to set up notation we will first define Poincar\'e duality groups. For further details, we refer the reader to \cite{bieri} and references therein. 

\begin{definition} 
Let $R$ be a commutative ring with unity. A group $G$ is said to be a duality group of dimension $n$ over $R$ if there is an $RG$ module $D_G$ such that for all $k \in \mathbb{Z}$ and for all $RG$ modules $L$, one has the following natural isomorphisms (referred to as duality isomorphisms). $$H^k(G;L) \cong H_{n-k} (G; D_G \otimes_R L)$$ 
\end{definition}

Here, $G$ acts diagonally on the tensor product. The module $D_G$ is called the dualising module of $G$. It is clear that the cohomological dimension over $R$ of such a group is at most $n$. Moreover, taking the module $L$ to be the induced module $RG$ in the duality isomorphism and applying the Eckmann-Shapiro Lemma, we find that $H^k(G,RG)= 0$ for all $k\neq n$ and $H^n(G,RG)= D_G$. This implies that $G$ is a group of type FP over $R$; however, much more is true as the following Theorem shows. 

\begin{theorem}\cite[Theorem 9.2]{bieri} A group $G$ is a duality group of dimension $n$ over $R$ if and only if the following three conditions hold. 
\begin{enumerate}
	\item $G$ is of type FP over $R$
	\item $H^k(G, RG)=0$ for $k\neq n$
	\item $H^n(G,RG)$ is flat as an $R$-module. 
\end{enumerate}
\end{theorem} 

\begin{definition} A group $G$ is called a Poincar\'e Duality group of dimension $n$ over $R$ (or a $PD^n(R)$-group, for short) if $G$ is a duality group of dimension $n$ over $R$ and the dualising module $D_G$ is isomorphic to $R$. 
\end{definition} 

When $G$ is the fundamental group of a closed aspherical manifold $M$, then Poincar\'e duality holds with $R=\Z$ and in this case the dualising module is a trivial $G$-module if and only if $M$ is orientable. We will follow the convention of writing $PD^n$ for $PD^n(\mathbb{Z})$. In this paper we will consider $PD^n$ groups and $PD^n(\field)$ groups. Observe that every $PD^n$ group is also a $PD^n(\field)$ group. 

A group is a one-dimensional duality group over $\mathbb{Z}$ if and only if it is finitely generated free. Consequently, a group $G$ is a $PD^1$ group if and only if $G \cong \mathbb{Z}$. That every $PD^2$ group is a surface group is a deep theorem due to Bieri, Eckmann, Linnell and Muller. It is conjectured that every finitely presented $PD^n$ group is the fundamental group of a closed aspherical manifold. 

The geometric and algebraic end invariants which play a crucial role in this paper may be defined in terms of the cohomology groups of a group with exotic coefficients. In order to define them we will need the following modules.

Let $\mathcal{P}G$ denote the collection of all subsets of $G$. Then, $\mathcal{P}G$ is an $\field$-vector space with respect to the operation of symmetric difference. One checks that $\mathcal{P}G$ is also a $G$ module. Moreover, $\mathcal{P}G \cong$ $\coind ^G _1 \field$. For  a subgroup $H<G$ we denote by $\mathcal{F}_H (G)$ the $\field G$-module $\ind ^G_H \mathcal{P}H$:

\[
\mathcal{F}_H (G)= \{A \subseteq G: A \subseteq HF \textrm{\ for some finite set F} \}.
\] 

Similarly the power set $\mathcal{P}(H\backslash G)$ of $H \backslash G$ and the collection of finite subsets of $H\backslash G$, written $\mathcal{F}(H
\backslash G)$ are $\field[G]$ modules. In fact, $\mathcal{P}(H\backslash G)$ $\cong \coind ^G _H \field$ and $\mathcal{F}(H \backslash G)$ $\cong \ind ^G _H \field$. 

We recall the following definitions from \cite{KrophollerRoller}:

\begin{definition}
For a subgroup $H<G$ the algebraic end invariant is defined as $$\tilde{e}(G,H):=\dim_{\field} \left(\mathcal{F}_H(G) \backslash \mathcal{P}G \right) ^G$$ and the geometric end invariant is defined as \[e(G,H)=\dim_{\field} \left(\mathcal{F}(H\backslash G)) \backslash \mathcal{P}(H\backslash G) \right)^G.\] 
\end{definition}

If $H$ has finite index in $G$, then coinduction coincides with induction and therefore $e(G,H)=0=\tilde{e}(G,H)$. If on the other hand, $H$ is of infinite index, then one can study the long exact sequence in cohomology corresponding to $0\rightarrow \mathcal{F}_H (G) \rightarrow \mathcal{P}G \rightarrow \mathcal{F}_H(G) \backslash \mathcal{P}G  \rightarrow 0$ and an easy computation yields the formula $\tilde{e}(G,H) = 1 + \dim_{\field} H^1(G, \mathcal{F}_H(G)).$

\medskip

We now have all the  notation required to state our splitting theorem for $PD^n$ groups. 

\begin{PDtheorem}\label{dualitytheorem}
Let $G$ be an orientable $PD^{n}$ group and $H$ be an orientable $PD^{n-1}$ subgroup of $G$. If every action of $H$ on a CAT(0) cube complex has a global fixed point, in particular, if $H$ satisfies Kazhdan's property (T), then $G$ splits over a Poincar\'e duality subgroup $C$ containing $H$ as a finite index subgroup. 
\end{PDtheorem}

\begin{proof} Let $G$ and $H$ be as in the statement of the theorem. As $PD^k$ groups are also $PD^k(\field)$ groups, we may work over $\field$. Since the dualising module $H^n(G, \field G) \cong \field$ is trivial the duality isomorphism gives  $H^k(G;\mathcal{F}_H(G)) \cong H_{n-k} (G; \mathcal{F}_H(G))$ for all $k \in \mathbb{Z}$. Therefore, $$H^1(G, \mathcal{F}_H(G)) \cong H_{n-1}(G,\mathcal{F}_H(G))\cong H_{n-1}(G, \textrm{Ind}^G_H (\mathcal{P}H)) \cong H_{n-1}(H, \mathcal{P}H),$$ where the last isomorphism is given by the classical Eckmann-Shapiro Lemma.
  
Using duality isomorphisms for $H$, we get $H_{n-1}(H, \mathcal{P}H )\cong H^0(H, \mathcal{P}H)\cong \field$. As $\tilde{e}(G,H) = 1 + \dim_{\field} H^1(G, \mathcal{F}_H(G))$, we deduce that $\tilde{e}(G,H)= 2$. We now invoke Lemma 2.5 of \cite{KrophollerRoller} to get a subgroup $H'$ of at most index $2$ in $H$ such that $e(G,H')$= $\tilde{e}(G,H)$=2. 

Applying Sageev's construction \cite[Theorem 2.3]{Sageev} we obtain a $\CAT(0)$ cube complex $X$ such that $G$ acts essentially on $X$ and $H'$ is the stabiliser of an oriented codimension-1 hyperplane $J$. As $H'$ has finite index in the group $H$  the action of $H'$ on the CAT(0) cube complex $J$ has a global fixed point. One now extracts from the fixed point of the action, a proper $H'$ almost invariant subset $B$ of $G$ such that $H'BH'=B$ \cite[Lemma 2.5]{Sageev}. 

The singularity obstruction $S_B(G,H')$, introduced in \cite{N}, is defined as the collection $\{g \in G\ :\ gB \cap B \neq \emptyset,\ gB^c\cap B \neq \emptyset,\ gB\cap B^c\neq \emptyset \textrm{ and } gB^c\cap B^c \neq \emptyset  \}$, where $B^c$ $=G\backslash B$. We now apply  \cite[Lemma 4.17] {Kropholler} as follows to deduce that  the subgroup $K_g=H'\cap gH'g^{-1}$ satisfies $\tilde{e}(G,K_g)\geq 2$. 

\begin{itemize}
\item[] When $g\in (B^*)^{-1}\cap B^*$ then $gB\cap B$ is a $K_g$-almost invariant subset.
\item[] When $g\in (B^*)^{-1}\cap B$ then $gB\cap B^*$ is a $K_g$-almost invariant subset.
\item[] When $g\in B^{-1}\cap B^*$ then $gB^*\cap B$ is a $K_g$-almost invariant subset.
\item[] When $g\in B^{-1}\cap B^*$ then $gB^*\cap B^*$ is a $K_g$-almost invariant subset.
\end{itemize}

We will use this to show that the elements of $S_B(G,H')$ all lie in the commensurator of $H'$, allowing us to apply the generalised Stallings' theorem from \cite{N}.

\noindent \textbf{Claim} For any $PD^n(\field)$ group $\Gamma$ with subgroup $\Gamma'$, if $\tilde{e}(\Gamma, \Gamma') \geq 2$ then $\textrm{cd}_{\field} \Gamma'$ $=n-1$. 

Let $\Gamma, \Gamma'$ be as in Claim. If $\tilde{e}(\Gamma,\Gamma') \geq 2$, then $\Gamma'$ has infinite index in $\Gamma$ and so by Strebel's theorem $\textrm{cd}_{\field} \Gamma' \leq n-1$. We will show that in fact we have equality.  Suppose not and $\textrm{cd}_{\field} \Gamma'$ $\leq n-2$. 

As before $\tilde{e}(\Gamma,\Gamma') \geq 2$ implies that $H^1(\Gamma, \textrm{Ind}^{\Gamma}_{\Gamma'}\mathcal{P}\Gamma')$ is non-zero. However, by duality for $\Gamma$ and the Eckmann-Shapiro Lemma, we have $H^1(\Gamma, \textrm{Ind}^{\Gamma}_{\Gamma'}\mathcal{P}\Gamma')$ $\cong$ $H_{n-1}(\Gamma', \mathcal{P}\Gamma')$. As $\textrm{cd}_{\field} \Gamma'$ $\leq n-2$, there is a projective resolution $P$ of $\field$ by $\field \Gamma'$-modules of length $n-2$. By definition, $H_{n-1}(\Gamma', \mathcal{P}\Gamma'))$ is the $(n-1)$-th homology of the complex $P \otimes_{\field\Gamma'} \mathcal{P}\Gamma'$. Clearly, the latter vanishes, contradicting the fact that $H^1(\Gamma, \textrm{Ind}^{\Gamma}_{\Gamma'}\mathcal{P}\Gamma')$ is non-zero. This proves the Claim. 

Returning to the proof of Theorem \ref{dualitytheorem} we deduce from the Claim that if $g \in S_B(G,H')$ then $K_g$ has the same cohomological dimension as $H'$ and as $H'^g$. It follows, again by Strebel's theorem that $K_g$ has finite index in both. Hence, $g$ lies in the commensurator $\textrm{Comm}_G(H')$ of $H'$ as required. Therefore by Theorem B of \cite{N}, $G$ splits over a subgroup $C$ commensurable with $H'$ and hence with $H$. More precisely, $G$ is either a non-trivial amalgamated free product $G\cong A*_C B $ or an HNN extension $A*_C$ such that $C$ and $H$ are commensurable. 

Now consider the action of $G$ on the Bass-Serre tree for the splitting. Since $H$ is commensurable with an edge stabiliser, and $G$ acts with no edge inversions,  the subgroup $H$ fixes a vertex $v$ so up to conjugation (and switching the roles of $A, B$ if necessary in the amalgamated free product case) we may assume that $H$ lies in the vertex stabiliser $A$. Having taken a conjugate of the picture to arrange for $H<A$ it is still the case that $H$ is commensurable with an edge stabiliser $C^g$ for some $g$, but, \emph{a priori}  it is not clear that $H$ is commensurable with $C$ itself. Suppose that $C^g$ fixes an edge $e$ so that $H\cap C^g$ fixes both $e$ and $v$. It then fixes the first edge $e'$ on the geodesic from $v$ to $e$. So the stabiliser of this edge is commensurable with $H$. Up to conjugation within $A$ this stabiliser is $C$, so we may express $G$ as an amalgamated free product $G\cong A*_C B $ or as an HNN extension $A*_C$ with $H<A$ and so that $H$ and $C$ are commensurable.

By \cite[Theorem V.8.2]{dicksdunwoody} we obtain a $PD^n$ pair $(A, \Omega)$, where $\Omega$ is the set of cosets of $C$ in $A$. There are two types of $PD^n$ pairs, the $I$-bundle type where $\Omega $ has two elements and the general type where $\Omega$ is infinite. 

As remarked in \cite[Section 2.1]{KrophollerRoller}, if $(A,\Omega)$ is of general type, then the subgroup $C$ is self-commensurating in $A$. If $(A,\Omega)$ is the I-bundle type and $A$ acts trivially on $\Omega$ then $A=C$. So $C$ is self-commensurating in $A$ in both these cases. As commensurable subgroups have the same commensurators and $H<A$ we conclude in both these cases that $H$ is a finite index subgroup of $C$ as required. 

It remains to consider the situation when $(A,\Omega)$ is of twisted I-bundle type i.e. $A$ acts non-trivially on $\Omega$. Here it is no longer evident that $C$ is self-commensurating in $A$ and the proof that $H$ is a subgroup of $C$ is slightly different. Recall that $\Omega$ is a 2-element set and $A$ acts nontrivially on $\Omega$ so that the stabiliser of any point in $\Omega$ is isomorphic to $C$. This implies that $C$ has index $2$ in $A$ and by \cite[Proposition VIII.10.2]{brown}, $A$ is a $PD^{n-1}$ group. We argue that in this case $A$ is non-orientable, using the following observation. 
\medskip

\noindent \textbf{Observation} If $(\Gamma,\Omega)$ is an orientable $PD^n$ pair of $I$-bundle type in which $\Gamma$ is an orientable $PD^{n-1}$ group then $\Gamma$ acts trivially on $\Omega$.

The proof given here of the observation above is due to Jonathan Cornick and Peter Kropholler, and we are grateful to them for allowing us to include it. 
Consider the long exact sequence of cohomology for the pair with coefficients in the integral group ring $\Z \Gamma$. Cohomology groups with these coefficients inherit an action of $\Gamma$. We obtain the following short exact sequence
\[
0\to \cohom{n-1}{\Gamma}{\Z \Gamma}\to\ext{n-1}{\Z \Gamma}{\Z\Omega}{\Z \Gamma}\to\cohom {n}{(\Gamma,\Omega)}{\Z \Gamma}\to0. 
\]
The orientability of the pair $(\Gamma,\Omega)$ implies $\cohom n{(\Gamma,\Omega)}{\Z \Gamma}$ is isomorphic to the trivial module $\Z$. The orientability of $\Gamma$ implies that $\cohom{n-1}\Gamma{\Z \Gamma}$ is isomorphic to the trivial module as well. From this it follows that the action of $\Gamma$ on the middle group $\ext{n-1}{\Z \Gamma}{\Z\Omega}{\Z \Gamma}$ is either trivial or of infinite order.

Let $H$ be the stabiliser of one of the points of $\Omega$. Let $K$ be the unique maximal orientable $PD^{n-1}$ subgroup in $H$. The action of $K$ on $\ext{n-1}{\Z \Gamma}{\Z\Omega}{\Z \Gamma}$ is trivial and $K$ has index at most $4$ in $\Gamma$. From this it follows that $\Gamma$ acts trivially on $\ext{n-1}{\Z \Gamma}{\Z\Omega}{\Z \Gamma}$. In particular, $\Gamma$ acts trivially on $\Omega$. This establishes the observation. 

Returning  again to the proof of Theorem \ref{PD}, we apply the observation with $\Gamma=A$. As the pair $(A, \Omega)$ is of twisted $I$-bundle type, $A$ must be non-orientable, so it has a non-trivial dualising module $D$. Proposition VIII.10.2 of \cite{brown} says that any $PD^{n-1}$ subgroup of $A$ inherits its dualising module by restriction from the dualising module of $A$, so $A$ has a unique maximal orientable $PD^{n-1}$ subgroup, $A^+$. Since $A$ is not orientable $A^+$ is a proper subgroup and since $C$ is orientable $C<A^+$. Finally since the index $[A\mathop:C]=2$ we conclude that $A^+=C$. Since $H$ is also orientable, $H<C$ as required.

\end{proof} 

\section{Prerequisites from Surgery Theory} \label{prelim2}

\emph{In this section (and in particular in  Lemmas \ref{belowmid} and \ref{middim} below), for a smooth (respectively, PL or topological)  manifold, \emph{submanifold} means a smooth (respectively, PL locally flat or topological locally flat) submanifold.}

Let $M, N$ be manifolds as in the statement of Theorem \ref{main} and $j:N\rightarrow M$  a $\pi_1$-injective map. Then $\pi_1(N), \pi_1(M)$ satisfy the hypotheses of Theorem \ref{dualitytheorem} providing a splitting of $\pi_1(M)$ as either an amalgamated free product $\pi_1(M)=A\mathop{*}_C B$ or as an HNN extension $A\mathop{*}_C$ where $\pi_1(N)<C$ is a subgroup of finite index. In \cite{cappell} Cappell provided tools to geometrise such a splitting. In particular he proved that if $C$ is square root closed in $\pi_1(M)$ then there is a closed aspherical embedded submanifold $i:N'\hookrightarrow M$ such that $i$ induces the inclusion of $C$ in $\pi_1(M)$ and, by Van Kampen's theorem, induces the splitting of $\pi_1(M)$. Asphericity then yields the homotopy commutative Figure \ref{diagram2} of the introduction.

The idea of Cappell's proof is to use surgery techniques to build a  cobordism from $M$ to a manifold $M'$ which does split in the required way. This provides the homotopy equivalences $h,h'$ required by Theorem \ref{main} and Cappell's square root closed hypothesis on $C$ can be used to promote the homotopy equivalences to a homeomorphism $M\cong M'$, bypassing the need for $\pi_1(M)$ to satisfy the Borel conjecture. Since we prefer not to invoke either the Borel conjecture or the square root closed hypothesis in our statement of Theorem \ref{main} it is necessary to unpack the proof of Cappell's splitting theorem.

As usual, surgery below the middle dimension requires no additional hypotheses, as is captured by the following lemma. We  have amended the notation to fit our situation:

\begin{lemma}\cite[Lemma I.1]{cappell} \label{belowmid}
Let $Y$ be an $(n+1)$ dimensional Poincar\'e complex and $X$ a codimension-$1$  sub-Poincar\'e complex with trivial normal bundle in $Y$ and with $\pi_1(X) \rightarrow \pi_1(Y)$ injective. Let $M$ be an $(n+1)$ dimensional closed manifold with $f:M \rightarrow Y$ a homotopy equivalence, $n \geq 5$. Assume we are given  $m < (n-1)/2$; then $f$ is homotopic to a map, which we continue to call $f$, which is transverse regular to $X$ (whence $N'=f^{-1}(X)$ is a codimension-1 submanifold of $M$) and such that  the restriction $f|_{N'}:N' \rightarrow X$ induces isomorphisms $\pi_i(N') \rightarrow \pi_i(X)$, $i \leq m$.
\end{lemma}

Lemma \ref{middim} describes the obstruction to carrying out surgery in the middle dimension which by Lemma \ref{aspherical} is all that is then required. It is carried by the surgery kernels $K_i(N')$ defined in \cite[Section I.4]{cappell}, and the projective class group $\widetilde{K_0}(C)$ appearing in Lemma \ref{middim} below. As recorded in the proof of \cite[Lemma I.2]{cappell}, if $\pi_i(f^{-1}(X)) \rightarrow \pi_i(X)$ is an isomorphism then $K_i(N')=0$. 

\begin{lemma}[Lemma II.1 from \cite{cappell}] \label{middim} 
Let $n=2k$, $M$ be a closed manifold and $Y$ a Poincar\'e complex of dimension $(n+1)$. Assume we are given $X$ a sub-Poincar\'e complex) of dimension $n$ of $Y$ with trivial normal bundle and with $C=\pi_1(X)< \pi_1(Y)$. Assume further that $f:M \rightarrow Y$ is a homotopy equivalence transverse regular to $X$ with, writing $N'=f^{-1}(X)$, $N'$ connected and $\pi_1(N') \rightarrow \pi_1(X)$ an isomorphism and $K_i(N')=0$ , $i < k$. Then letting $K_k(N')=P\oplus Q$ denote the decomposition of $\mathbb{Z}C$ modules defined in \cite{cappell} I.4, 
\begin{enumerate} 
\item $K_k(N')$ is a stably free $\mathbb{Z}C$ module and $[P]=-[Q]$. Moreover,\\ in Case I, $[P]\in \ker{(\widetilde{K_0}(C) \rightarrow \tilde{K}_0(G_1) \oplus \tilde{K}_0(G_2))}$, and \\ in Case II, $[P] \in \ker{(\widetilde{K_0}(C) \rightarrow \widetilde{K_0}(J))}$.
\item Any finite set of elements of $P$ (respectively, $Q$) can be represented by embedded disjoint framed spheres in $N'$ for $k >2$. The intersection pairing of $K_k(N')$ is trivial when restricted to $P$ (resp; $Q$) and $Q \cong P^*$. Thus, $[P]=-[P^*]$.
\item\label{vanishingP} If $[P]=0$, $f$ is homotopic to a map $f'$ with $N'= f'^{-1}(X)$ $\rightarrow X$ $k$-connected and so that, abusing notation  by writing $K_k(N')=P \oplus Q$ for the decomposition of \cite{cappell} I.4, $P$ and $Q$ are free $\mathbb{Z}C$-modules. 
\end{enumerate} 
\end{lemma}

Applying Lemma \ref{belowmid} up to dimension $n/2-1$ ensures the vanishing of the surgery kernels $K_i(N')$ while hypothesis \ref{K_0} of Theorem \ref{main} ensures the vanishing of $\widetilde{K_0}(C)$  so that $[P]=0$ as required by condition \ref{vanishingP} of Lemma \ref{middim}.

Cappell's lemmas along with his Nilpotent Normal Cobordism Construction allow us to realise the splitting obtained in Theorem \ref{PD} by an embedded submanifold $N'$ in a homotopic manifold $M'$ which is aspherical up to the middle dimension. Given that we are working with Poincar\'e duality groups the following  elementary lemma from homotopy theory guarantees that this is sufficient to ensure that $N'$ is aspherical. 

\begin{lemma} \label{aspherical}
Let $X^n$ be a closed, orientable manifold and let $k$ be the largest integer less than or equal to $\frac{n}{2}$. Suppose the universal cover $\tilde X$ of $X$ is $k$-connected and that $G=\pi_1(X)$ is a duality group of dimension $n$. Then $\pi_i(X)=\{0\}$ for all $i\geq 2$.
\end{lemma}

\begin{proof} Let $X$ be a manifold as in the statement of the Lemma. Suppose that $\pi_i(X)\neq\{0\}$ for some $i\geq 2$ and choose the smallest such $p$. Evidently, $\tilde X$ is $(p-1)$-connected. By the Hurewicz Theorem,  $H_i(\tilde X)=0$ for all $i= 1,\ldots , (p-1)$ and $\pi_p(\tilde X) \cong $ $H_p(\tilde X)$. Observe that $k+1 \leq p\leq n$. Duality for orientable non-compact manifolds implies that $H_p(\tilde X)$ $\cong H_c^{n-p}(\tilde X)$. Here, $H_c^*(\tilde X)$ refers to the cohomology with compact supports for $\tilde X$. We claim that $H_c^{n-p}(\tilde X) \cong H^{n-p}(G, \mathbb{Z}G)$.

Recall that $H_c^i(\tilde X)$ $= \varinjlim_K H^i(\tilde X, \tilde X - K)$, as $K$ varies over compact subsets of $\tilde X$ and hence is the $i$-th homology of the complex $\varinjlim_K $ $\hom(C(\tilde X, \tilde X - K),\mathbb{Z})$. Note that every element of $\varinjlim_K $ $\hom(C(\tilde X, \tilde X - K),\mathbb{Z})$ is represented by a module homomorphism from $ C(\tilde X, \tilde X - K)$ to $\mathbb{Z}$ for some compact $K$. 
 
As $G$ acts freely and properly on the universal cover $\tilde X$ with quotient $X$, we can consider the partial augmented singular chain complex $C_p(\tilde X) \rightarrow C_{p-1}(\tilde X) \rightarrow \ldots \rightarrow C_0(\tilde X) \rightarrow \mathbb{Z} \rightarrow 0$. Since $\tilde X$ is $(p-1)$-connected this is a partial free $G$-resolution of $\mathbb{Z}$ and may be completed to a projective resolution $P$. The action of $G$ is proper and so for $i < p$, the cohomology of the co-complex $\hom_G(C(\tilde X), \mathbb{Z}G)$ is the cohomology of $X$ with local coefficients in $\mathbb{Z}G$. Hence this is the $i$-th cohomology of $G$ with $\mathbb{Z}G$ coefficients. 

We now show that for $i <p$, $\hom_G(C_i(\tilde X), \mathbb{Z}G)$  $\cong \varinjlim_K \hom(C_i(\tilde X, \tilde X - K),\mathbb{Z})$, where $K$ varies over compact subsets of $\tilde X$. Let $F \in \hom_G(C_i(\tilde X), \mathbb{Z}G)$. As a $\mathbb{Z}$-module map, $F: C_i(\tilde X)$ $\rightarrow \mathbb{Z}G$ is induced by the assignment $\sigma \mapsto \sum_g f_g(\sigma).g$, where the sum is finite and $f_g$ is a module homomorphism from $C_i(\tilde X)$ to $\mathbb{Z}$. A straightforward computation shows that $F$ is a $G$-module homomorphism if and only if $f_g(\sigma)=f_1(g^{-1}\sigma)$ for all $g \in G$. Note that $f_1$ is zero outside of a compact subset $K$ of $\tilde X$ and hence we have a map $\hom_G(C_i(\tilde X), \mathbb{Z}G)$  $\rightarrow \varinjlim_K \hom(C_i(\tilde X, \tilde X - K),\mathbb{Z})$ given by $F \mapsto f_1$. One checks that this is an isomorphism with inverse coming from the prescription $f\mapsto (F:\sigma \mapsto \sum_g f(g^{-1}\sigma).g)$ and hence $H_c^{n-p}(\tilde X) \cong H^{n-
p}(G, \mathbb{Z}G)$. 

Finally, putting all the isomorphisms together, $\pi_p(X) \cong H^{n-p}(G,\mathbb{Z}G)$. As $0 \leq n-p \leq k-1$ and $G$ is a duality group of dimension $n$ group, $H^{i}(G,\mathbb{Z}G)=0$ except possibly when $i=n$, so in particular $H^{n-p}(G,\mathbb{Z}G)=0$   which contradicts our hypothesis on $\pi_p(X)$. 
\end{proof}

\section{Proof of Theorem \ref{main}}\label{proof}
\begin{proof} Let $M$ and $N$ be as in the statement of Theorem \ref{main}. Set $G=\pi_1(M)$ and $H=\pi_1(N)$. Then $G$ is an orientable $PD^{n+1}$ group and $H$ is an orientable $PD^n$ subgroup of $G$ and by Theorem \ref{PD}, $G$ splits over a Poincar\'e duality subgroup $C$ containing $H$ as a subgroup of finite index. 

It is well known that a group $G$ splits over a subgroup $H$ (in the sense of our convention) if and only if $G$ acts on an unbounded tree $T$ with a single edge orbit, and so that $H$ is an edge stabiliser. It follows from the unboundedness of the action that vertex and edge stabilisers of the action are all of infinite index in $G$. When the splitting is a non-trivial amalgamated free product $G=A\mathop{*}\limits_C B$, the subgroups $A,B$ and $C$ are all of infinite index in $G$ and so, by Strebel's theorem \cite[Theorem 3.2]{Strebel}, the cohomological dimension of each of $A,B, C$ is at most $n$. Since $A, B$  contain $C$ which has cohomological dimension $n$, all three have cohomological dimension equal to $n$. Similarly, in the case of the HNN decomposition, the subgroups $A, C$ have cohomological dimension $n$. Since $n>2$, the Eilenberg-Ganea theorem \cite{EilenbergGanea}, applies and the cohomological dimensions of $A$, $B$ and $C$ are equal to the geometric dimensions of $A$, $B$ and $C$ 
respectively.

We now carry out the standard mapping cylinder construction of an Eilenberg MacLane space for $G$. When the group $G$ splits as an amalgamated free product $G=A*_C B$ let $X_A$, $X_B$ and $X_C$ denote the $n$-dimensional $K(\pi,1)$ complexes for $A$, $B$ and $C$. Note that these complexes may be taken to be simplicial complexes. Since $G, C$ are both of type FP, we can apply the Mayer Vietoris sequence, to see that  the groups $A$ and $B$ are also of type FP so we may choose the  complexes to be of finite type.

Let $\phi_A: X_C\rightarrow X_A$ and $\phi_B:X_C\rightarrow X_B$ be maps inducing the inclusions $C\hookrightarrow A$ and $C\hookrightarrow B$. We construct the mapping cylinders $M_A, M_B$ of these maps and making the standard identifications yields a $K(\pi,1)$ for the group $G$ which we will denote by $Y$.  In the case when $G$ splits as an HNN extension we carry out a similar process building a $K(\pi, 1)$ from $X_A$ and $X_C\times [-1,1]$ by glueing the two ends $X_C\times \{ \pm 1\}$ to $X_A$ using maps which induce the two inclusions of $C$ into $A$ defining the HNN extension. In this case again we denote the $K(\pi,1)$ by $Y$.

Note that $Y$ is a Poincar\'e complex, and moreover, the Poincar\'e subcomplex $X_C\times \{0\}\subset Y$ cuts the neighbourhood  $X_C\times (-1/2, 1/2)\subset Y$  into two components so that the normal bundle of $X_C$ in $Y$ is trivial (see Introduction of \cite{cappell}). By construction, the composition $ X_C\rightarrow X_C\times\{0\}\hookrightarrow Y$ is $\pi_1$-injective. Applying Lemma \ref{belowmid}, we may replace $f$ (up to homotopy) by a map, which we continue to call $f$, which is transverse regular to $X_C\times\{0\}$  and such that  the restriction of $f$ to $N'=f^{-1}(X_C\times \{0\}) \rightarrow X_C$ induces isomorphisms $\pi_i(N') \rightarrow \pi_i(X_C)$, for all $i\leq (n/2-1)$. In particular, $\pi_i(N')=0$ for all $i = 2,\dots,(n/2-1)$. Note that the transverse regularity of $f$ ensures that the pre-image $N'$ is a codimension-$1$ submanifold of $M$. Our aim is to  further modify the map $f$ so that it remains transverse regular and so that its restriction to $N'$ induces an isomorphism on 
all homotopy groups, thus making $N'$ aspherical. 
 
Recall that $n$ is even so that $n=2k$ with $k \geq 3$. As recorded in the proof of Lemma I.2 of \cite{cappell}, the isomorphisms $\pi_i(N') \rightarrow \pi_i(X_C)$, for all $i\leq (k-1)$ ensure that each of the modules $K_i(N')$ for $i \leq k-1$ vanishes. 

Since $C$ is a torsion free finite extension of $H$, condition \ref{K_0} of Theorem \ref{main} ensures the vanishing of $\widetilde{K_0}(C)$.  Invoking Lemma \ref{middim}, we replace $f$ by a map $f'$ homotopy equivalent to $f$ which is transverse regular and such that $\pi_1(N')\cong C$ where $N'=f'^{-1}(X)$ and the restriction of $f'$ to $N'$ is $k$-connected. In addition, both $P$ and $Q$ in $K_k(N')=P \oplus Q$ are free $C$-modules. This means that one can perform Cappell's Nilpotent Normal Cobordism Construction on $M$, a process we will now describe. 

Note that cutting $M$ along $N'$ we get a decomposition $M=M_A \cup_{N'} M_B$ in Case I and $M=M_A/\{N'_A \cong N' \cong N'_B \}$ in Case II. The covering of $M$ corresponding to the image of $\pi_1(X) \rightarrow \pi_1(Y) \cong \pi_1(M)$ is labelled $\hat{M}$ and the universal covering of $M$ is labelled $\tilde{M}$. The  group  $C=\pi_1(N')$ acts by covering transformations on $\tilde M$, $M_L$ and $M_R$ (where $\tilde M= M_L \cup_{\tilde{N'}} M_R$) with quotients being $\hat{M}$, $M_l$ and $M_r$. Hence, $\hat{M}=M_l \cup_{N'} M_r$. Let $I=[0,1]$ and $I'=[-2,2]$. We select a tubular neighbourhood $N' \times I'$ of $N'$ in $M$ such that when the lift of $N'$ to $\hat{M}$ is extended to a lift of $N' \times I'$, we have $N' \times \{-2 \} \subset M_l$ and $N' \times \{2\} \subset M_r$. 

\begin{figure}[h] 
   \centering
   \includegraphics[width=6in]{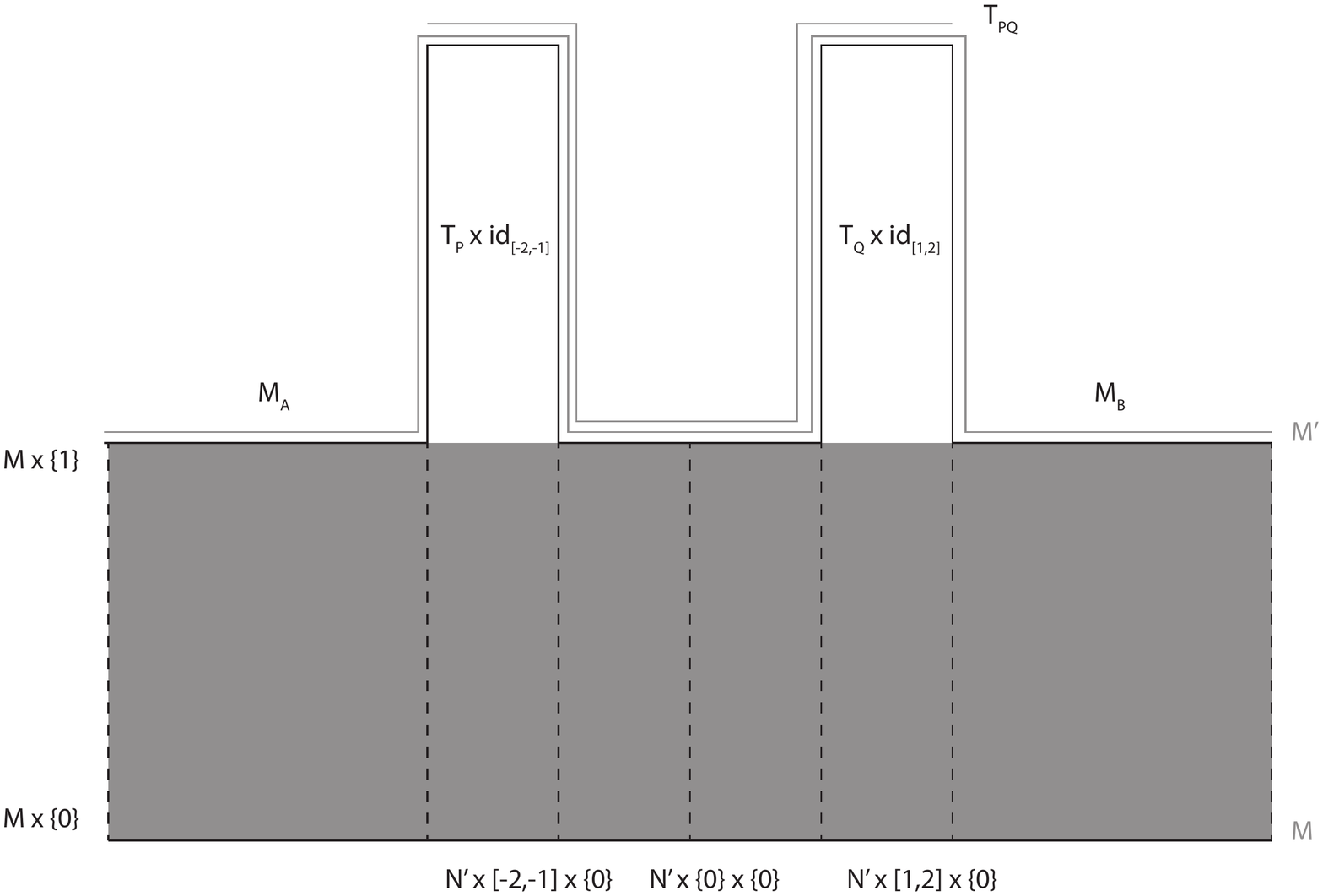} 
   \caption{Extending the trivial cobordism on $M$ via surgery along the spheres $\{X_i\}$ and $\{Y_j\}$.}
   \label{fig:cobordismextension}
\end{figure}

Let $\{x_i\}_{i=1}^{s}$ denote a $\pi_1(N')$-free basis for $P$ and let $\{y_i\}_{i=1}^{s}$ denote the dual basis for $Q$ under the intersection pairing of $K_k(N')$. We choose disjoint framed spheres  $\{X_i\}_{i=1}^{s}$ and $\{Y_i\}_{i=1}^{s}$  representing the bases for $P$ and $Q$ respectively. We can assume given the choice of bases, that for $i \neq j$, $X_i \cap Y_j = \emptyset$ and for any $i$, $X_i$ intersects $Y_i$ in a point. We kill the spheres $X_i$ by $k$-surgery on $N'$ to obtain a new manifold $N'_P$ yielding a cobordism $T_P$ of $N'$ with $N'_P$; similarly killing the spheres $Y_i$ by $k$-surgery we get a cobordism $T_Q$ of $N'$ with a new manifold $N'_Q$. Clearly, $\pi_1(N'_P) \cong C \cong \pi_1(N'_Q)$ and by Lemma \ref{aspherical} both $N'_P, N'_Q$ are aspherical and thus homotopy equivalent to $X_C$. 

Now consider the trivial cobordism $M \times I$. We will extend it to a cobordism from $M$ to $M'$ by applying the cobordism extension lemma, gluing $T_P \times [-2,-1]$ to $M \times \{1\}$ along $N' \times [-2,-1] \times \{1\}$ and glueing $T_Q\times [1,2]$ to $M \times \{1\}$ along $M \times [1,2]$. The boundary of the resulting manifold $T$ has two components, one being $M$ and the other, a new manifold $M'$ given as follows.\\
\noindent $M'= (M_A \cup_{N'} T_P) \cup_{N'_P} T_{PQ} \cup_{N'_Q}(T_Q \cup_{N'} M_B), \textrm{ where }$\\
\noindent $T_{PQ}= (N'_P \times [-2,-1]) \cup_{N'_P} T_P \cup_{N'} (N' \times [-1,1]) \cup_{N'} (T_Q \cup_{N'_Q} N'_Q \times [1,2]).$

Observe that the manifold $T_{PQ}$ sits naturally inside $M'$ and gives a cobordism between $N'_P$ and $N'_Q$. Moreover not only are $N'_P$ and $N'_Q$ homotopy equivalent to $X_C$, the cobordism $T_{PQ}$ constructed by filling the spheres $\alpha_i$ and $\beta_i$ is also homotopy equivalent to $X_C$. The original homotopy equivalence $f: M \rightarrow Y$ extends to a map $F: T \approx M\cup T_P \times [-2,-1]\cup T_Q\times [1,2]\rightarrow Y\times I$ which restricts first to a homotopy equivalence $M' \rightarrow Y$ (\cite[Lemma II.5]{cappell}) and further to the homotopy equivalence $T_{PQ} \rightarrow X_C$.

Hence we obtain a manifold $M'$ which is homotopy equivalent to $M$ and split along $X_C \subset Y$ via the embedded submanifold we continue to call $N'$. Recall also that $H$ is a finite index subgroup of $C$. Hence, there is a finite cover $N_0$ of $N'$ with fundamental group $H$. As $N_0$ and $N$ are aspherical manifolds with isomorphic fundamental groups there exists a homotopy equivalence $h'$ between them. This completes the proof of Theorem \ref{main}.

\end{proof}

\section{An obstruction result: Corollary \ref{firstcorollary}}\label{application} 

Recall the following classical fact: if $M^{4d+1}$ is a smooth manifold such that the first Betti number $b_1(M)$ of $M$ is $0$ and $N^{4d}$ has non-zero signature then there are no immersions of $N$ into $M$. This follows using Hirzebruch's signature theorem: since $f$ is a codimension-1 immersion, $f^*\: H^{4d}(M, \mathbb{Q}) \rightarrow H^{4d}(N, \mathbb{Q})$ maps the Hirzebruch L-class $L_d(M)$ onto $L_d(N)$. It follows from Poincar\'e duality that $H^{4d}(M, \mathbb{Q})$ is isomorphic to the torsion free part of $H_1(M)$. Therefore, the vanishing of $b_1(M)$ forces $L_d(N)$ to be zero. However Hirzebruch's signature theorem says that $L_d(N)$ is equal to the signature of $N$. This applies, for example, to show that there are no codimension-1 immersions of an orientable quaternionic hyperbolic or Cayley hyperbolic manifold  $N^{4d}$ into a smooth orientable aspherical manifold $M^{4d+1}$ with $b_1(M)=0$. 

Appealing instead to the non-vanishing of Pontryagin numbers, we obtain the following generalisation as a corollary to Theorem \ref{main}, which obstructs the existence of $\pi_1$-injective continuous maps rather than immersions. Note that while our hypotheses on $N$ imply that it satisfies the Borel conjecture so that we can take the homotopy equivalence $h'$ to be a homeomorphism, we do not assume in the statement that the target manifold satisfies the Borel conjecture, nor do we assume Cappell's square root closed hypothesis, since in the proof we are applying Theorem \ref{main} in its most general form.

\begin{corollary2}\label{Pontryagin}
Let $M^{4d+1}$ be a closed, orientable, aspherical, smooth manifold  such that $d\geq2$ and the first Betti number $b_1(M)$ is zero. Let $N^{4d}$ be a closed, orientable, aspherical, smooth manifold with at least one non-zero Pontryagin number such that $\pi_1(N)$ is either word hyperbolic or a CAT(0) group and satisfies Kazhdan's property (T). Then there are no $\pi_1$-injective continuous maps $f\: N \rightarrow M$.
\end{corollary2}

\begin{proof}
Carrying out the Cappell surgery arguments in the smooth category, the proof of Theorem \ref{main} shows that there is a 2-sided embedded smooth submanifold $N'$ in a smooth manifold $M'$  homotopy equivalent to $M$ that realises the splitting. Furthermore, since $\pi_1(N)$ satisfies the Borel conjecture, the map $p\circ h':N\rightarrow N'$ is a covering map. Now if $N'$ is separating, write $M_1$ for one of the connected components of the manifold obtained from cutting $M'$ along $N'$. Then $M_1$ is a smooth orientable manifold with boundary $N'$ and hence $N'$ bounds orientably. This means that all Pontryagin numbers for $N'$ must vanish. On the other hand we are assuming that $N$ has at least one non-zero Pontryagin number (for the case when $N$ is quaternionic or Cayley hyperbolic, see \cite{LafontRoy}). Since Pontryagin numbers vary multiplicatively with degree on covering maps the Pontryagin numbers of $N'$ are also non-zero and so $N'$ cannot bound orientably. This means that $N'$ cannot be separating 
and so intersection with $N'$ yields an element of infinite order in the first cohomology of $M'$. By duality and the universal coefficients theorem the first Betti number $b_1(M')$  and hence $b_1(M)$ is non-zero. 
\end{proof}


\begin{thebibliography}{}
\bibitem{BL} Bartels,~A. and L\"uck,~W., The Borel conjecture for hyperbolic and CAT(0) groups. Ann. of Math. (2) 175 (2012), no. 2, 631--689.
\bibitem{bieri} Bieri,~R., Homological dimension of discrete groups, Queen Mary College Mathematics Notes, London (1976). 
\bibitem{brown} Brown,~K., Cohomology of groups. Graduate Texts in Mathematics 87, Springer-Verlag, New York-Berlin (1982).
\bibitem{cappell} Cappell,~S., A Splitting Theorem for Manifolds, Invent. Math. 33, 69--170 (1976). 
\bibitem{Davis} Davis,~M., The Geometry and Topology of Coxeter Groups, London Mathematical Society Monographs, Princeton University Press (2008). 
\bibitem{dicksdunwoody} Dicks,~W.~ and Dunwoody,~M.~J., Groups acting on graphs, CUP (1982). 
\bibitem{dixonmortimer} Dixon,~J.~D. and Mortimer,~B., Permutation Groups, Graduate Texts in Mathematics 163, Springer (1996). 
\bibitem{Dunwoody-Sageev} Dunwoody, M.~J.~and Sageev,~M.~E., JSJ-splittings for finitely presented groups over slender groups. Inventiones Mathematicae, 135, (1), 25--44 (1999). (doi:10.1007/s002220050278).
\bibitem{Dunwoody-Swenson} Dunwoody,~M.~J.~ and Swenson,~E.~L., The algebraic torus theorem. Inventiones Mathematicae, 140, (3), 605--637 (2000). (doi:10.1007/s002220000063).
\bibitem{EilenbergGanea} Eilenberg,~S.~ and Ganea,~T., On the Lusternik-Schnirelmann Category of Abstract Groups, Ann. of Math. 65, 517--518 (1957).
\bibitem{Farb} Farb,~B, The quasi-isometry classification of lattices in semi-simple Lie groups, Mathematical Research Letters 4, 705--717 (1997).
\bibitem{FarrellJones} Farrell,~F.~T.~and Jones,~L.~E., Topological rigidity for compact non-positively curved manifolds, Differential geometry: Riemannian geometry (Los Angeles, CA, 1990), 229--274, Amer. Math. Soc., Providence, RI, (1993). 
\bibitem{KahnMarkovic} Kahn,~J.~and Markovic,~V., Immersing almost geodesic surfaces in a closed hyperbolic three manifold, Ann. of Math. (2) 175 (2012), no. 3, 1127--1190.
\bibitem{Kropholler} Kropholler,~P.~H, A Group Theoretic Proof of the Torus Theorem, Geometric Group Theory Volume 1, Eds. G.A. Niblo and M. Roller, LMS Lecture Notes Series 181, CUP (1993) 138--158. 
\bibitem{KrophollerRoller} Kropholler,~P.~H.~and Roller,~M.~A., Splittings of Poincar\'e duality groups III, J. London Math. Soc. 39, 271--284 (1989).
\bibitem{LafontRoy} Lafont,~J. and Roy,~R., A note on the characteristic classes of non-positively curved manifolds, Expo. Math. 25, 21--35 (2007).
\bibitem{Mok} Mok,~N., Siu,~Y.~T.~and Yeung,~S., Geometric superrigidity. Invent. Math. 113, 57--83 (1983).
\bibitem{N} Niblo,~G.~A, The singularity obstruction for group splittings, Topology Appl. 119, 17--31 (2002).
\bibitem{NibloReeves} Niblo,~G.~A. and Reeves,~L.~D.,  Groups acting on CAT(0) cube complexes, Geom. Topol. 1, 1--8 (1997).
\bibitem{Papakyriokopoulos} Papakyriakopoulos,~C., On Dehn's lemma and asphericity of knots. Ann. of Mathematics 66 (1), 1--26 (1957). doi:10.2307/1970113. JSTOR 1970113
\bibitem{Reid} Reid,~A.,  A non-Haken hyperbolic 3-manifold finitely covered by a surface bundle, Pacific J. Math. 167, 163--182 (1995).
\bibitem{Sageev} Sageev,~M.,  Co-dimension 1 subgroups and splittings of groups, J. Algebra, 189, 377--389 (1997).
\bibitem{Scott} Scott,~G.~P., Ends of pairs of groups, Journal of Pure and Applied Algebra, 11 (1977) 179--198.
\bibitem{Scott-Swarup} Scott,~G.~P.~and Swarup,~G.~A., Regular neighbourhoods and canonical decompositions for groups. As\'erisque, (289):vi+233, (2003).
\bibitem{Strebel} Strebel,~R., A remark on subgroups of infinite index in Poincar\'e duality groups, Comment. Math. Helv. 52, 317--324 (1977).
\bibitem{Waldhausen} Waldhausen,~F., On the determination of some bounded 3-manifolds by their fundamental groups alone, Proc. Internat. Sympos. on Topology and its Applications, Herieg-Novi, Yugoslavia, Aug.~ 25-31,~1968,~Belograd,~1969,~ 331--332.
\end{thebibliography}
\end{document}